\documentclass{amsart}

\usepackage{amscd}
\usepackage{tikz-cd}
\usepackage{xcolor}

\usepackage{amssymb}
\usepackage[all]{xy}
\usepackage{hyperref}

\newtheorem{thm}{Theorem}

\newtheorem{lem}[thm]{Lemma}
\newtheorem{cor}[thm]{Corollary}

\newtheorem{prop}[thm]{Proposition}

\newtheorem{conj}[thm]{Conjecture}
   
\theoremstyle{definition}
\newtheorem{defn}[thm]{Definition}

\newtheorem{say}[thm]{}
\newtheorem{exmp}[thm]{Example}

\newtheorem{rem}[thm]{Remark}          

\newtheorem*{ack}{Acknowledgments}      

\newtheorem{defn-thm}[thm]{Definition--Theorem}  
\newtheorem{defn-lem}[thm]{Definition--Lemma}  

\theoremstyle{remark}
\newtheorem{claim}[thm]{Claim}

\renewcommand{\c}[0]{{\mathbb C}}  

\renewcommand{\o}[0]{{\mathcal O}} 
\newcommand{\z}[0]{{\mathbb Z}}

\renewcommand{\r}[0]{{\mathbb R}}

\newcommand{\q}[0]{{\mathbb Q}}
\newcommand{\map}[0]{\dasharrow}
\newcommand{\qtq}[1]{\quad\mbox{#1}\quad}

\newcommand{\pic}[0]{\operatorname{Pic}}

\newcommand{\supp}[0]{\operatorname{Supp}}

\newcommand{\im}[0]{\operatorname{im}}

\newcommand{\ex}[0]{\operatorname{Ex}}    
\newcommand{\diff}[0]{\operatorname{Diff}}

\newcommand{\rdown}[1]{\lfloor{#1}\rfloor}

\newcommand{\tsum}[0]{\textstyle{\sum}}

\newcommand{\coeff}[0]{\operatorname{coeff}}

\def\into{\DOTSB\lhook\joinrel\to}

\def\loccoh#1.#2.#3.#4.{H^{#1}_{#2}(#3,#4)}

\DeclareMathAlphabet{\mathchanc}{OT1}{pzc}%
                                {m}{it}

\newtheorem*{sketch}{Sketch of the proof}

\usepackage[all]{xy}\xyoption{dvips}

\title[Vanishing theorems for threefolds]
{Vanishing theorems for threefolds in characteristic $p>5$} 
\author{Fabio Bernasconi and J\'anos Koll\'ar} 

\address{École Polytechnique Fédérale de Lausanne,  MA B3 515 (Bâtiment MA), 1015, Lausanne, Switzerland} 
\email{fabio.bernasconi@epfl.ch}

\address{Princeton University, Princeton NJ 08544-1000, USA}

\email{kollar@math.princeton.edu}

\begin{document}


\begin{abstract}
	We prove Grauert-Riemenschneider--type vanishing theorems for excellent dlt threefolds pairs whose closed points have perfect residue fields of positive characteristic $p>5$. Then we discuss applications to dlt singularities and to Mori fiber spaces of threefolds.
\end{abstract}
\maketitle

\section{Introduction}

The Grauert-Riemenschneider vanishing theorem says that if $g\colon Y\to X$ is a proper birational morphism over $\c$ and $Y$ is smooth then
$R^ig_*\omega_Y=0$ for $i>0$. This is known to fail in positive characteristic, even for 3-folds.  
However, Grauert-Riemenschneider vanishing holds if $Y, X$ are both regular excellent schemes by  \cite[Theorem 1]{CR15}, or if $Y$ is an excellent Cohen-Macaulay scheme and $X$ has
rational singularities by \cite[Theorem 1.4]{kov-rtl}. 

In many applications one would need a similar vanishing theorem where
$\omega_Y$ is replaced by some other line bundle. 
Our main technical result, Theorem~\ref{dlt.chr7.gr.cor}, gives such a vanishing  for excellent dlt 3-folds whose closed points have perfect residue fields of positive characteristic $p >5$.

Then in Section~\ref{g-r.cons.sec} we use Theorem~\ref{dlt.chr7.gr.cor} to derive local rationality properties of dlt 3-fold pairs in positive and mixed characteristic. Finally we use the same techniques to obtain liftability and rationality results for Mori fiber spaces in dimension 3 over perfect fields of positive characteristic; see Section~\ref{s-app}.

\begin{defn}[G-R vanishing]\label{g-r.prop.defn}
	Let $(X, \Delta)$ be a pair, where  $X$ is a normal, excellent scheme
	with a dualizing complex and  $\Delta$ is a boundary on $X$ (that is, an
	$\mathbb{R}$-divisor whose coefficients are in $[0,1]$). We assume from
	now on  that log resolutions of singularities exist for the  pairs that we
	work with.
	
	We say that  {\it(strong)  G-R vanishing} holds {\it over}   $(X, \Delta)$
	if
	the following is  satisfied for every  log resolution $g \colon (X',
	E+g^{-1}_*\Delta)\to (X, \Delta)$, where $E$ is the exceptional divisor of
	$g$.
	\begin{enumerate}
		\item[]  Let  $D'$ be a Weil $\z$-divisor  and $\Delta'$ an  effective
		$\r$-divisor on $X'$. Assume that      $g_*\Delta'\leq \Delta$,
		$\rdown{\ex(\Delta')}=0$ (where  $\ex(\Delta')$ denotes the
		$g$-exceptional part of $\Delta'$), and
		\begin{enumerate}
			\item (for G-R vanishing)   $D'\sim_{g,\r} K_{X'}+\Delta'$,
			\item (for strong G-R vanishing)   $D'\sim_{g,\r}
			K_{X'}+\Delta'+(\mbox{$g$-nef $\r$-divisor})$.
		\end{enumerate}
		Then  $R^ig_*\o_{X'}(D')=0$ for $i>0$.
	\end{enumerate}
	We check in Section~\ref{bir.inv.subsect} that if  G-R vanishing holds
	for one log resolution of $X$, then it holds for every log resolution.
\end{defn}

\begin{rem}
	G-R vanishing is simpler to handle since  the condition
	$D'\sim_{g,\r} K_{X'}+\Delta'$ is preserved by contractions and flips.
	By contrast, being $g$-nef is not preserved.
	However, the two versions might be  almost equivalent, though we can not formulate a precise statement or conjecture.
\end{rem}

If $X$ is essentially of finite type over a field of characteristic 0, then
strong G-R vanishing is a special case of the general Kodaira-type vanishing theorems; see \cite[2.68]{km-book}. Recently these vanishing theorems have been extended to general excellent $\mathbb{Q}$-schemes by Murayama; see \cite{murayama2021relative}.
Strong G-R vanishing also holds over 2-dimensional, excellent schemes by
\cite{lip-rs}; see \cite[10.4]{kk-singbook}.
In particular, if $X$ is any normal excellent scheme  then
the support of $R^ig_*\o_{X'}(D')=0$ has codimension $\geq 3$ for $i>0$.

However, G-R vanishing is known to fail for 3-folds in every positive characteristic, as shown by affine cones over smooth projective surfaces violating the Kodaira vanishing theorem.
Thus we need some restrictions on the singularities of the base for G-R vanishing to hold in positive and mixed characteristic.
The following is our main result:

\begin{thm} \label{dlt.chr7.gr.cor}
	Let $(X, \Delta)$ be a  3-dimensional dlt pair,
	where $X$ is an excellent scheme with a dualizing complex whose closed points have perfect residue fields of characteristic $p>5$.
	Then G-R vanishing holds over $(X, \Delta)$. 
\end{thm}

In  Theorem~\ref{dlt.chr7.gr.cor} the assumption on the characteristics is optimal, as shown by the examples in \cite{CT19,Ber17,ABL20}.

\begin{sketch}
	To prove Theorem \ref{dlt.chr7.gr.cor}, we consider a log resolution $g \colon Y \to (X, \Delta)$ and we run a suitable MMP on $Y$ which will end with $(X, \Delta)$. 	
	Since in Section \ref{bir.inv.subsect} we show that G-R vanishing holds on snc threefold pairs, we are left to show that G-R vanishing is preserved under the steps of the MMP. This is done in Proposition \ref{g-r.desc.contr.lem} (resp. Proposition \ref{g-r.desc.flip.lem}) in the case of divisorial contractions (resp. flips). The case of divisorial contractions is where we need the assumptions on the residue fields as we use the following vanishing theorem for surfaces of del Pezzo type proved in \cite[Theorem 1.1]{ABL20}:
	\begin{thm}
		Let $k$ be a perfect field of characteristic $p>5$. 
		Let $X$ be a surface of del Pezzo type over $k$. 
		Let $D$ be a Weil divisor on $X$ and suppose that there exists an effective $\mathbb{Q}$-divisor $\Delta$ such that $(X, \Delta)$ is a klt pair
		and $D-(K_X+\Delta)$ is big and nef. Then $H^i(X, \mathcal{O}_X(D))=0$ for $i>0$.
	\end{thm}
\end{sketch}

\begin{ack} We thank Christopher~D.~Hacon, Liam~Stigant, Jakub~Witaszek, Chenyang~Xu,  Takehiko~Yasuda and the anonymous referees for helpful comments and corrections.
	
	FB was supported by the NSF under grant number DMS-1801851 and by a grant from the Simons Foundation;
	Award Number: 256202. JK was supported by the NSF under grant number
	DMS-1901855.
\end{ack}

\section{Preliminaries}

In this section we gather various results that we use. Most of these have appeared in the literature, but not stated in the form or generality that we need. 

We refer to \cite{kk-singbook} and \cite[Section 2.5]{mixed-char-3-mmp} for the basic definitions in birational geometry and of the singularities appearing in the Minimal Model Program for excellent schemes.

\subsection{A restriction short exact sequence}

The following two results  are proved, though  not explicitly stated, in \cite[Section 3]{HW19}.

\begin{lem} \label{flex.KV.ind.lem.1}  
	Let $(X, S+\Delta)$ be an excellent dlt pair where $S$ is a prime divisor with normalization $\nu \colon \bar S \to S$.   Then, for every Weil $\z$-divisor class $D_X$ on $X$, there are
	\begin{enumerate}
		\item  a Weil  $\z$-divisor class  $D_{\bar S}$ on $\bar S$,
		\item  a $\q$-divisor $0\leq \Theta_D\leq \diff_{\bar S}(\Delta)$, and
		\item a natural left exact sequence 
		$$
		0\to \omega_X(D_X) \to  \omega_X(S+D_X)\stackrel{r_S}{\longrightarrow} \nu_*\omega_{\bar S}(D_{\bar S})\to 0,
		\eqno{(\ref{flex.KV.ind.lem.1}.3.1)}
		$$
	\end{enumerate}
	such that
	\begin{enumerate}\setcounter{enumi}{3}
		\item $D_{\bar S}\sim_{\mathbb{Q}} D|_{\bar S}+\Theta_D$,
		\item  the sequence (\ref{flex.KV.ind.lem.1}.3.1) is exact along the regular locus of $S$, and
		\item if $\omega_X(D_X)$ is $S_3$ along $S$, then (\ref{flex.KV.ind.lem.1}.3.1) is exact. 
	\end{enumerate}
\end{lem}

\begin{proof}
	Note that $S$ is regular in codimension one by \cite[2.31]{kk-singbook}, but it need not be normal in positive characteristic (see \cite{CT19, Ber19}).
	
	We start by  replacing $D_X$ with a linearly equivalent divisor not containing $S$. 
	In this way $D_X|_{\bar S}$ is a well-defined $\q$-divisor on ${\bar S}$.
	We now have the exact sequence
	$$ 0\to \omega_X(D_X) \to  \omega_X(S+D_X)\stackrel{r_S}{\longrightarrow} \mathcal{Q}\to 0,
	$$
	where $\mathcal{Q}$ is a sheaf supported on $S$.
	Note that $\mathcal{Q}$ is a torsion free sheaf of rank one on $S$ by \cite[2.60]{kk-singbook} and  there is a natural injection into its  $S_2$-hull $\mathcal{Q} \hookrightarrow \mathcal{Q}^{(**)}$ (if $S$ is normal, the $S_2$-hull equals the reflexive hull; see \cite{k-coherent} for the general case.)
	
	Let $\mathcal{Q}|_{S_{\text{reg}}}$ be the restriction of $\mathcal{Q}$ to the regular locus of $S$. Since $S$ is regular in codimension one, $\mathcal{Q}|_{S_{\text{reg}}}$ extends naturally to a reflexive divisorial sheaf $\overline{\mathcal{Q}}$ on $\bar S$. Since $\nu_*\overline{\mathcal{Q}}$ is $S_2$ and coincide with $\mathcal{Q}^{(**)}$ on a big open set, we conclude they are isomorphic.
	
	We are left to show that $\overline{\mathcal{Q}}$ is isomorphic to $\o_{\bar S}(K_{\bar S}+D_{\bar S}),$ where $D_{\bar S}:= \lfloor D_X|_{\bar S} \rfloor$.
	It is enough to prove this after localization at  generic points of $\supp \overline{\mathcal{Q}}$. We may thus assume that $\dim X=2$, and then
	$\pi \colon Y \to (X, S+ \Delta)$ has a thrifty log resolution \cite[2.79]{kk-singbook}. The proof of \cite[Proposition 3.1]{HW19} now carries over. 
	
	If $\omega_X(D_X)$ is $S_3$, we conclude that the sequence (\ref{flex.KV.ind.lem.1}.3.1) is exact because $\mathcal{Q}$ is $S_2$ by \cite[2.60]{kk-singbook}.	
\end{proof}


We will use  sequences as in  (\ref{flex.KV.ind.lem.1}.3.1) to lift
vanishing statements from $S$ to $X$, but we need a stronger form of vanishing of Kawamata-Viehweg type on $S$. 

\begin{defn}\label{flex.KV.defn}  Let $(X, S+\Delta)$ be an excellent dlt pair and $g\colon X\to Z$ a proper  morphism.  By {\it K-V vanishing} 
	for $g \colon (X, S+\Delta)\to Z$ we mean the following property.
	\begin{enumerate}
		\item $R^ig_*\o_X(D)=0$ for $i>0$, for every 
		Weil divisor $D$  such that  
		$D\sim_{g, \r} K_X+ S+\Delta'+L$ for some
		$0\leq\Delta'\leq\Delta$, where   $L$ is $g$-ample.
	\end{enumerate}
	
	K-V vanishing is known to hold if $Z$ is essentially of finite type over a field of characteristic 0, if $\dim X=2$ and $g$ has relative dimension $\leq 1$, or if
	$(X, S+\Delta)$ is a log del~Pezzo surface pair over a perfect field and the characteristic is $\geq 7$ by
	\cite{ABL20}. 
\end{defn}

\begin{prop} \label{flex.KV.ind.lem}  
	Let $(X, S+\Delta)$ be an excellent dlt pair, $g \colon X\to Z$ a proper  morphism and
	$D$ a  Weil $\z$-divisor  on $X$.
	Assume that
	\begin{enumerate}
		\item $D\sim_{g,\r} K_X+ S+\Delta'+L$ for some
		$0\leq\Delta'\leq\Delta$, 
		\item  $\o_X(D-mS)$ is $S_3$ along $S$ for every $m\geq 1$, and 
		\item   $-S$ is $g$-nef.
	\end{enumerate}
	Assume further that one of the following is satisfied.
	\begin{enumerate}\setcounter{enumi}{3}
		\item  $L$ is an ample $\q$-Cartier divisor and K-V vanishing holds for 
		$g|_{\bar S}$, or
		\item $L$ is a nef $\q$-Cartier divisor, $g|_{\bar S}$ is birational and strong G-R vanishing holds for $g|_{\bar S}$.
	\end{enumerate}
	Then $R^ig_*\o_X(D)=0$ for $i>0$ near $g(S)$.
\end{prop}

\begin{proof}  Set  $J_m:=\im[\o_X(D-mS)\to \o_X(D)]$.
	Applying Lemma \ref{flex.KV.ind.lem.1} to $D_m:=D-K_X-(m+1)S$ we get a short exact sequence:
	\[0 \to \o_X(D-(m+1)S) \to \o_X(D-mS) \to \nu_*\o_{\bar S}(G_m) \to 0, \]
	where $G_m \sim_{\r} -mS|_{\bar{S}}+D|_S-\Delta_m$ for some  $0 \leq \Delta_m \leq \diff_{\bar{S}}(\Delta')$ and $\nu \colon \bar{S} \to S$ is the normalisation morphism. 
	Thus by the snake lemma we get the following short exact sequence:
	\[ 0 \to \nu_*\o_{\bar S}(G_m) \to \mathcal{O}_{X}(D)/J_{m+1} \to \mathcal{O}_{X}(D)/J_m \to 0.  \]
	
	Consider the induced exact sequence for $i>0$:
	\[R^ig_* \nu_*\o_{\bar S}(G_m) \to R^ig_*(\o_X(D)/J_{m+1}) \to R^ig_*(\o_X(D)/J_{m}).\]
	The first term vanishes by assumptions (4--5) because $G_m \sim_{g|_{\bar{S}}, \r} K_{\bar{S}}+\diff_{\bar{S}}(\Delta') -\Delta_m +L|_{\bar{S}}-mS|_{\bar{S}}$, so we conclude by induction that $R^ig_*(\o_X(D)/J_{m})=0$ for all $m \geq 0$.
	Thus  $R^ig_*\o_X(D)=0$ for $i>0$ near $g(S)$
	by the formal function theorem.
\end{proof}

\subsection{Fujita transform and CM criteria} \label{s-fuj-transf}

One can not reasonably define the pull-back of an arbitrary  Weil $\z$-divisor  by a birational morphism. Nonetheless, there are two  replacements---considered in \cite{fujita85}---that are very useful for vanishing theorems. We will use these  in Section~\ref{g-r.cons.sec}.

\begin{defn}[Fujita transforms]\label{fuj.tr.say}
	Let $\pi\colon X\to Y$ be a proper, birational morphism of normal schemes.
	Let  $D$  be a Weil $\z$-divisor on $Y$ and $\Delta_D$ an   $\r$-divisor.
	(Eventually it will be a  boundary, but for now this is not needed). Assume that $D+\Delta_D$ is $\r$-Cartier. 
	
	The {\it Fujita transform} of $(D,\Delta_D)$ is the unique pair
	$\bigl(\pi^F(D), \Delta_{X,D}\bigr)$ such that
	\begin{enumerate}
		\item $\pi^{\rm F}(D)$ is a Weil $\z$-divisor on $X$ such that $\pi_*\bigl(\pi^{\rm F}(D)\bigr)=D$,
		\item $\Delta_{X,D}$ is  an   $\r$-divisor on $X$ such that $\pi_*(\Delta_{X,D})=\Delta_D$,
		\item $\rdown{\ex(\Delta_{X,D})}=0$, and 
		\item $\pi^{\rm F}(D)+ \Delta_{X,D}\sim_{\pi, \r} 0$.
	\end{enumerate}
	{\it Warning.} The notation is slightly misleading since  $\pi^{\rm F}(D)$ also depends on $\Delta_D$. In our applications, $\Delta_D$ will be fixed.
	\medskip
	
	Assume next that $K_Y+\Delta$  is $\r$-Cartier  and $0\leq \Delta_D\leq \Delta$.
	Set $\Delta^c_D:=\Delta - \Delta_D$.
	The {\it K-twisted Fujita transform} of $(D,\Delta_D)$ is the unique pair
	$\bigl(\pi^{\rm KF}(D), \Delta^c_{X,D}\bigr)$ such that
	\begin{enumerate}
		\item[(1')] $\pi^{\rm KF}(D)$ is a Weil $\z$-divisor on $X$ such that $\pi_*\bigl(\pi^{\rm KF}(D)\bigr)=D$,
		\item[(2')] $\Delta^c_{X,D}$ is  an   $\r$-divisor on $X$ such that $\pi_*(\Delta^c_{X,D})=\Delta^c_D$,
		\item[(3')] $\rdown{\ex(\Delta^c_{X,D})}=0$, and 
		\item[(4')] $\pi^{\rm KF}(D)\sim_{\pi,\r} K_X+ \Delta^c_{X,D}$.
	\end{enumerate}
\end{defn}	
Note that   the Fujita transform  and the K-twisted Fujita transform are both functors.

The  key observations of \cite{fujita85} 
are the following three claims. 

\begin{claim} \label{c.pushforward.fujita}
	$\pi_*\o_X\bigl(\pi^{\rm F}(D)\bigr)=\o_Y(D)$. 
\end{claim}
\begin{proof}  
	As $-\pi^{F}(D) \sim_{g, \r} \Delta_{X,D}$, by Property 3 it is elementary to deduce that $\o_{Y}(\pi_* \pi^F(D)) = \pi_* \o_{X}(\pi^{F}(D))$ (we refer to \cite[7.30]{kk-singbook} for details). We thus conclude by Property 1.
\end{proof}

To get the second main property, 
write  $\pi^*(K_Y+\Delta)=K_X+\Delta_X$ where  $\pi_*(\Delta_X)=\Delta$.
Then 
$$
\pi^{\rm KF}(D)-\pi^{\rm F}(D)\sim_{\pi,\r} -\Delta_X+\Delta_{X,D}^{c}+\Delta_{X,D}.
$$
Here $\Delta_{X,D}^{c}$ and $\Delta_{X,D}$ are effective by definition. If all divisors appear in
$\Delta_X $ with coefficients $<1$, then all divisors appear in
$\pi^{\rm KF}(D)-\pi^{\rm F}(D)$ with coefficients $>-1$. 
Since $\pi^{\rm KF}(D)-\pi^{\rm F}(D)$ is a  $\z$-divisor, 
it is then effective.
Thus  we have proved the following.

\begin{claim} \label{c.klt.Kfujtransf}
	If $(Y, \Delta)$ is klt  or if $(Y, \Delta)$ is dlt and $\pi$ is thrifty, then
	$\pi^{\rm KF}(D)-\pi^{\rm F}(D)$ is an effective $\pi$-exceptional divisor. 
	Thus $\pi_*\o_X\bigl(\pi^{\rm KF}(D)\bigr)=\o_Y(D)$.
\end{claim}

We get the third property by comparing the definitions. 

\begin{claim} \label{c.klt.Kfujtransf.3}  $\big(K_X-\pi^{\rm F}(D)\big)$ and $\pi^{\rm KF}(D)$ satisfy the assumptions of G-R vanishing  
	(\ref{g-r.prop.defn}.1.a). Thus if G-R vanishing holds over $(Y, \Delta)$ 
	for some $\Delta\geq\Delta_D$, 
	then $R^i\pi_*\o_X\bigl(K_X-\pi^{\rm F}(D)\bigr)=0$ and
	$R^i\pi_*\o_X\bigl(\pi^{\rm KF}(D)\bigr)=0$ for $i>0$. \qed
\end{claim}

To get a CM criterion out of the above claims, we use the method of two spectral sequences  (see \cite{k-dep} or \cite[7.27]{kk-singbook}) combined with the duality  between local cohomology and higher direct images.
Note that \cite[10.44]{kk-singbook} states the needed duality  for locally free sheaves. In our case one needs to use the Cohen-Macaulay version as in 
\cite[5.71]{km-book}. 
The end result is  the following.

\begin{thm}[CM-criterion] \label{cm.crit.thm}   
	Let $\pi \colon X\to Y$ be a  proper birational morphism between integral schemes. 
	Let $D$ be  a Weil $\z$-divisor  and $\Delta$ an   $\r$-divisor on $Y$
	such that $D+\Delta$ is $\r$-Cartier.
	Assume that
	\begin{enumerate}
		\item  $X$ is normal and CM, 
		\item $\o_X\bigl(\pi^{\rm F}(D)\bigr)$ and $\o_X\bigl(\pi^{\rm KF}(D)\bigr)$ are CM.
		\item $R^i\pi_*\o_X\bigl(K_X-\pi^{\rm F}(D)\bigr)=0$ for $i>0$, 
		\item $ R^i\pi_*\o_X\bigl(\pi^{\rm KF}(D)\bigr)=0$ for $i>0$, and 
		\item $\pi^{\rm KF}(D)-\pi^{\rm F}(D)$ is effective.
	\end{enumerate}
	Then $ \o_Y(D)$ is CM. 
\end{thm}

Note that assumptions (1--2) are automatic if $X$ is a regular scheme. 

\begin{proof}
	Let $p \in Y$ be a closed point and let $W:= \pi^{-1}(p)$.
	We have the following commutative diagram of local cohomology groups:
	\begin{equation*}
		\xymatrix{
			H^i_p(Y, \pi_*\mathcal{O}_X(\pi^{\rm KF}(D)))  \ar[r]^{\alpha_i} &  H^i_{W}(X, \mathcal{O}_X(\pi^{\rm KF}(D))) \\
			H^i_p(Y, \pi_*\mathcal{O}_X(\pi^{\rm F}(D))) \ar[u]^{\beta_i} \ar[r] &H^i_W(X, \mathcal{O}_X(\pi^{\rm F}(D))) \ar[u]
		}
	\end{equation*}
	Here $\alpha_i$ is an isomorphism by hypothesis (4) and the Leray spectral sequence for local cohomology, and $\beta_i$ is an isomorphism by hypothesis (5).
	By \cite[10.44]{kk-singbook} we have $H^i_{W}(X, \mathcal{O}_X(\pi^{\rm F}(D))) \simeq (R^{n-i}\pi_*\mathcal{O}_X(K_X-\pi^{\rm F}(D)))_p$, which vanishes by (3) for $i<n$.
	From these  we deduce that $H^i_p(Y, \pi_*\mathcal{O}_X(\pi^{\rm F}(D)))=0$, and we conclude by noting that $\pi_*\mathcal{O}_{X}(\pi^{\rm KF}(D))=\pi_*\mathcal{O}_X(\pi^{\rm F}(D))=\mathcal{O}_Y(D)$ by hypothesis (4) and Claim~\ref{c.pushforward.fujita}.
\end{proof}

\subsection{Birational invariance of higher direct images}\label{bir.inv.subsect} 

\begin{conj} \label{fuj.tr.say.conj}  Let $Y$ be a  normal excellent scheme, 
	$D$ a Weil $\z$-divisor on $Y$ and $\Delta_D$ a  Weil $\r$-divisor  whose coefficients are in $[0,1]$. Assume that $D+\Delta_D$ is $\r$-Cartier. 
	Let $\pi\colon X\to Y$ be a proper, birational log resolution of $(Y, D+\Delta_D)$, and $\pi^{\rm KF}(D)$ the K-twisted Fujita transform.
	
	Then  $\mathbf{R} \pi_*\o_X\bigl(\pi^{\rm KF}(D)\bigr)$ is independent of $X$.
\end{conj}

It is possible that this can be proved by the  methods of \cite{kov-rtl}.
Here we have a more modest aim, to prove that resolution of singularities implies Conjecture~\ref{fuj.tr.say.conj}.  Let us first state what we need about the existence of resolutions.

\begin{say}[Resolution assumption]\label{res.ass}  Fix an integral scheme $Y$. 
	Let $g \colon X\to Y$ be a proper, birational morphism.
	Assume that $(X, D)$ is  an snc pair. Let
	$g'\colon X'\to Y$ be a proper, birational morphism. 
	Then there is a sequence of blow-ups
	$$
	X_m\stackrel{\tau_m}{\to} X_{m-1}\to \cdots \to X_1\stackrel{\tau_1}{\to} X_0=X
	$$
	such that
	\begin{enumerate}
		\item the center of each blow-up is a non-singular subvariety that has simple normal crossing with the total transform of $D$, (these are called {\it  non-singular blow-ups}),  and
		\item the induced map  $X_m\map X'$ is a morphism.
	\end{enumerate}
\end{say}

\medskip

The usual Leray spectral sequence argument reduces Conjecture~\ref{fuj.tr.say.conj} to the case when
$(Y, D+\Delta_D)$ is an snc pair. 

\begin{conj} \label{fuj.tr.say.conj.reg}  G-R vanishing holds over snc pairs. That is, using the notation of (\ref{fuj.tr.say.conj}), assume in addition that 
	$(Y, D+\Delta_D)$ is snc.
	Then  
	$$
	R^i \pi_*\o_X\bigl(\pi^{\rm KF}(D)\bigr)=0 \qtq{for} i>0.
	\eqno{(\ref{fuj.tr.say.conj.reg}.1)}
	$$
\end{conj}

\begin{thm}\label{fuj.tr.say.conj.reg.pf}
	Assume the resolution assumption \ref{res.ass}.
	Then Conjectures \ref{fuj.tr.say.conj} and \ref{fuj.tr.say.conj.reg} are true. In particular, the conjectures hold in the case where $Y$ is an excellent $\mathbb{Q}$-scheme or $Y$ is an excellent scheme of dimension $d \leq 3$.
\end{thm}

\begin{proof}
	By an elementary computation,  (\ref{fuj.tr.say.conj.reg}.1) holds
	if $\pi$ is a  non-singular blow-up.  Thus it also holds for compositions of
	non-singular blow-ups, more generally, for any birational map that is a composition of non-singular blow-ups and blow-downs. The Weak Factorization Conjecture asserts that every   birational map between non-singular varieties is such, but this is not known in positive characteristic. 
	We go around this using an unpublished argument of 
	Hironaka  (cf.\  \cite[p.153]{main-res-paper}). 
	
	To simplify notation, let $D_X$ denote the K-twisted Fujita transform of $D$ on $X$.
	By induction on $j\geq 1$ we prove that
	$$
	R^i\sigma_*\o_{X}\bigl(D_X\bigr)=0\qtq{for} j\geq i>0,
	\eqno{(\ref{fuj.tr.say.conj.reg.pf}.1)}
	$$
	for any birational morphism $\sigma \colon X\to X'$ between log resolutions of 
	$(Y, D+\Delta_D)$.
	
	For $j=1$, given any $\sigma \colon X\to X'$,
	we use our Resolution assumption~\ref{res.ass}  to get
	$$
	\tau\colon Z\stackrel{\pi}{\to} X \stackrel{\sigma}{\to} X',
	$$
	where $\tau$ is a composition of  non-singular blow-ups.
	
	We know that (\ref{fuj.tr.say.conj.reg}.1) holds for $\tau$. 
	The Leray spectral sequence 
	$$
	R^p\sigma_* (R^q\pi_*\o_Z\bigl(D_Z\bigr)) \Rightarrow R^{p+q}\tau_*\o_Z\bigl(D_Z\bigr)
	$$
	gives the inclusion $R^1 \sigma_*(\pi_*\mathcal{O}_Z(D_Z)) = R^1 \sigma_* (\mathcal{O}_X(D_X)) \hookrightarrow R^1 \tau_* \mathcal{O}_Z(D_Z)=0.$ So we prove \ref{fuj.tr.say.conj.reg.pf}.1 for $j=1$.
	
	Now assume \ref{fuj.tr.say.conj.reg.pf}.1 holds for $j$ and we prove it for $j+1$. Again by the Leray spectral sequence we have the inclusion 
	$$
	R^{j+1}\sigma_*\o_X\bigl(D_X\bigr)=R^{j+1}\sigma_*(\pi_*\o_Z\bigl(D_Z\bigr))\into R^{j+1}\tau_*\o_Z\bigl(D_Z\bigr)=0. 
	$$
	Thus (\ref{fuj.tr.say.conj.reg.pf}.1) also holds for $j+1$. 
	
	Finally, Assumption \ref{res.ass} is known to hold for varieties in characteristic 0 by \cite{main-res-paper} and for excellent
	threefolds \cite{CJS20}.
\end{proof}

\subsection{Rational singularities}

We recall the definition of rational singularities, following \cite{kov-rtl}.

\begin{defn} \label{d-rat}
	A scheme $X$ has \emph{rational singularities} if
	\begin{enumerate}
		\item $X$ is a normal, excellent, Cohen-Macaulay scheme admitting a dualising complex;
		\item  for every locally projective, birational morphism $\pi \colon 
		\widetilde{X} \to X$ where $\widetilde{X}$ is an excellent Cohen-Macaulay scheme, the natural morphism  $\o_X \to \mathbf{R}{\pi}_* \o_{\widetilde{X}} $ is an isomorphism.
	\end{enumerate} 
\end{defn}

If $X$ has a  resolution of singularities, then, by \cite[Corollary 9.11]{kov-rtl}, it is enough to check condition (2) in Definition~\ref{d-rat} for a resolution;  recovering the traditional notion of rational singularities.

In positive and mixed characteristics it is necessary to assume that $X$ is Cohen-Macaulay, but
in characteristic $0$ it can be deduced from the other properties as an application of the Grauert-Riemenschneider vanishing theorem.

\medskip

The following descent property is very similar to  \cite[Theorem~9.12]{kov-rtl}.

\begin{prop} \label{rationalsing}
	Let $\pi \colon Y \rightarrow X$ be a morphism of normal excellent schemes. Assume that 
	\begin{enumerate}
		\item $Y$ has rational singularities,
		\item the natural morphism  $\mathcal{O}_X \rightarrow \mathbf{R}\pi_* \o_Y$ splits in $D(X)$, \label{h-2}
		\item $X$ is a Cohen-Macaulay scheme admitting a dualising complex.
	\end{enumerate}
	Then $X$ has rational singularities.
\end{prop}

\begin{proof}
	The proof closely follows \cite[Theorem 1.1]{Kov00}. 
	By \cite[Corollary 9.11]{kov-rtl} it is sufficient to show that $X$ has pseudo-rational singularities, namely that for every normal scheme $\widetilde{X}$ and every projective birational morphism $\varphi \colon \widetilde{X} \rightarrow X$, the natural map $\varphi_* \omega_{\widetilde{X}} \to \omega_X$ is an isomorphism. 
	
	By the existence of a Macaulayfication (see \cite[Theorem 1.1]{Kaw00}) we can construct a projective birational morphism $\widetilde{Y} \to Y$ such that $\widetilde{Y}$ is Cohen-Macaulay and the following diagram commutes:
	\begin{equation*}
		\xymatrix{
			\widetilde{Y} \ar[r]^{\widetilde{\pi}} \ar[d]^{\psi} & \widetilde{X} \ar[d]^{\varphi} \\
			Y \ar[r]^{\pi} & X 
		}
	\end{equation*}
	
	Thus we have the following commutative diagram in $D(X)$:
	\begin{equation*}
		\xymatrix{
			\mathcal{O}_X \ar[r] \ar[d] & \mathbf{R} \pi_* \mathcal{O}_Y \ar[d] \\
			\mathbf{R}\varphi_*\mathcal{O}_{\widetilde{X}} \ar[r] & \mathbf{R} \varphi_* \mathbf{R}\widetilde{\pi}_* \mathcal{O}_{\widetilde{Y}}
		}
	\end{equation*}
	Since $Y$ has rational singularities and $\widetilde{Y}$ is Cohen-Macaulay, the composition
	\[ \mathbf{R} \pi_* \o_Y \to \mathbf{R} \varphi_* \mathbf{R}\widetilde{\pi}_* \mathcal{O}_{\widetilde{Y}} \simeq \mathbf{R} \pi_* \mathbf{R} \psi_* \o_{\widetilde{Y}} \simeq  \mathbf{R} \pi_* \o_Y \]
	is a an isomorphism. 
	Therefore by assumption (\ref{h-2}) we have a splitting in $D(X)$:
	\[ \mathcal{O}_X \rightarrow \mathbf{R} \varphi_* \mathcal{O}_{\widetilde{X}} \rightarrow \mathcal{O}_X. \]
	
	Applying $\mathbf{R} \mathcal{H}om _X( - , \omega_X^\bullet)$ and Grothendieck duality to the above sequence, we have the following splitting:
	\begin{equation}\label{split-omega}
		\omega_X^\bullet \rightarrow \mathbf{R} \varphi_* \omega^{\bullet}_{\widetilde{X}} \rightarrow \omega_X^{\bullet}.
	\end{equation} 
	Since $X$ is Cohen-Macaulay, we have $ \omega_X^\bullet \simeq \omega_X[-d]. $
	Considering the $(-d)$-th cohomology group in sequence  (\ref{split-omega}),
	we deduce that the composition
	\[ \omega_X \rightarrow \varphi_* \omega_{\widetilde{X}} \rightarrow \omega_X \]
	is an isomorphism. 
	Note that $\varphi_*\omega_{\widetilde{X}}$ is a torsion-free sheaf of rank one since $\omega_{\widetilde{X}}$ is so. 
	Therefore $\varphi_* \omega_{\widetilde{X}} \rightarrow \omega_X$ is an isomorphism.
	This means that $X$ has pseudo-rational singularities. 
\end{proof}

\section{Consequences of Grauert-Riemenschneider vanishing}\label{g-r.cons.sec}

We claim that those dlt pairs for which G-R vanishing holds satisfy all the other known local rationality properties that are known to hold in characteristic 0. This claim is not really new;  the arguments below are at least implicit in many
characteristic 0 papers, especially  \cite{k-dep},   
and in the positive characteristic works of
\cite{HW19, ABL20}.
However, the clearest  statements of such principles are in \cite{Kov00, kov-rtl};
see also \cite{Kov12}  and 
\cite[Secs.\ 2.5 and 7.3]{kk-singbook}  for more introductory treatments.

\begin{thm}\label{g-r.impl.other.thm}
	Let $(X, \Delta)$ be an excellent  dlt pair that admits a thrifty log resolution. Suppose that G-R vanishing holds over it. Then
	\begin{enumerate}
		\item $X$ is Cohen-Macaulay;
		\item Let $D$ be a  Weil $\z$-divisor on $X$ such that
		$D+ \Delta_D$ is $\r$-Cartier for some $0\leq \Delta_D\leq \Delta$. Then 
		$\o_X(D)$ is CM.
		\item $(X, B)$ is a rational pair (as in \cite[2.80]{kk-singbook})  for every $B\subset \rdown{\Delta}$, in particular $X$ has rational singularities;
		\item Log canonical centers of $(X, \Delta)$ are normal and  have rational singularities;
	\end{enumerate}
\end{thm}

All the above properties hold for excellent dlt pairs defined over $\mathbb{Q}$ by \cite{murayama2021relative}, and for dlt 3-folds as in  Theorem~\ref{dlt.chr7.gr.cor}. 
Note that, for dlt 3-folds pairs over perfect fields of characteristic at least 7, (1)  was proved  in \cite{HW19, ABL20}.

\begin{proof} 	If (2) holds then applying it to $D=\Delta_D=0$ we conclude that $X$ is CM.
	
	As for the proof of (2), let  $\pi \colon Y \to (X, \Delta)$ be a thrifty log resolution. Then we can apply Theorem \ref{cm.crit.thm} by Claims \ref{c.klt.Kfujtransf} and \ref{c.klt.Kfujtransf.3} to conclude that $\mathcal{O}_X(D)$ is CM.
	
	Let  $g \colon Y \to (X, \Delta)$ be a thrifty log resolution and let
	$B_Y$ denote the birational transform of $B$ and $\Delta'_Y$ the strict transform of $\Delta'=\Delta-B$.
	We write $$K_Y+B_Y+\Delta'_Y+F=g^*(K_X+\Delta)+E,$$ where $E$ and $F$ are both effective exceptional divisor without common exceptional components and $\lfloor{F \rfloor}=0$.
	Since $\lceil{E \rceil}-B_Y \sim_{g, \r} K_Y+\Delta'_Y+F+(\lceil{E\rceil} - E)$, we have that G-R vanishing applies both to $\mathcal{O}_Y(\lceil{E \rceil}-B_Y)$ and to $\omega_Y(B_Y)$, so the same proof of \cite[2.87]{kk-singbook} works in this setting, giving (3).

	We prove (4) by induction. Let $B_1, \dots, B_r$ be the irreducible components of $\lfloor{\Delta \rfloor}$. 
	We now prove that every irreducible component of $B_{i_1} \cap \cdots\cap B_{i_r}$ is normal and  has rational singularities. 
	If $B_i$ is an irreducible component of $\rdown{\Delta}$,
	applying (2)  to $\Delta_{B_i}:=B_i$ shows that $\o_X(-B_i)$ is CM.
	Now the sequence
	$$
	0\to \o_X(-B_i)\to \o_X\to \o_{B_i} \to 0
	$$
	shows that $\o_{B_i}$ is CM  (cf.\   \cite[2.60]{kk-singbook}). 
	By  \cite[2.31]{kk-singbook} $B_i$ is regular in codimension 1, hence normal. 
	Thus the pair $(B_i, \diff_{B_i}(\Delta-B_i))$ is dlt by easy adjunction \cite[4.8]{kk-singbook}.
	
	The following claim is crucial for the induction argument:
	\medskip
	
	{\it Claim \ref{g-r.impl.other.thm}.5.}
	Let $g \colon Y\to X$ be a proper morphism and $\{D_i:i\in I\}$ Cartier divisors on $Y$. Let $\mathcal{L}$ be an invertible sheaf on $Y$ and for $J\subset I$ set $D_J:=\cup_{i\in J}D_i$.  
	Assume that
	$R^mg_* \mathcal{L}(-D_J)=0$ for every $J\subset I$ and $m>0$.
	Then, for every $i\in I$,
	$$
	R^m(g|_{D_i})_*\bigl(\mathcal{L}(-{D_J})|_{D_i}\bigr)=0\qtq{for every}  J\subset I\setminus\{i\}\qtq{and} m>0.
	$$
	\begin{proof}
		Consider the short exact sequence $ 0 \to \mathcal{O}_Y(-D_i) \to \mathcal{O}_Y \to \mathcal{O}|_{D_i} \to 0$ and tensor it by $\mathcal{L}(-D_J)$ where $J \subset I \setminus\{i\}$. For $m>0$ we have the following exact sequence:
		$$ R^m g_*\mathcal{L}(-D_J) \to R^m (g|_{D_{i}})_*(\mathcal{L}(-D_J)|_{D_i}) \to R^{m+1} g_*\mathcal{L}(-D_J-D_i).$$
		We conclude by hypothesis that $R^m (g|_{D_{i}})_*(\mathcal{L}(-D_J)|_{D_i})=0$.
	\end{proof}
	
	To prove that an irreducible component $B_{ij}:=B_i \cap B_j$ of $\Gamma_i=\diff_{B_i}(\Delta-B_i)$ is CM, it is sufficient to show that $\o_{B_i}(-B_{ij})$ is CM. 
	By Claim \ref{c.klt.Kfujtransf.3} on $X$, naturality of the Fujita transforms and Claim \ref{g-r.impl.other.thm}.5 applied to $\mathcal{L}=\mathcal{O}_Y$ and $\left\{ D_i=g^{\rm KF}(-B_i)  \right\}$ (resp. $\mathcal{L}=\omega_Y$ and $\left\{ D_i=g^{\rm F}(-B_i)  \right\}$) the following higher direct images
	$$R^m (g|_{B_i'})_* \o_{B_i'}((g|_{B_i'})^{\rm KF}(-B'_{ij})) \text{ and } R^m (g|_{B_i'})_* \o_{B'_i} (K_{B'}-g^{F}(-B'_{ij})) \text{ for } m>0 $$
	vanish, so we can apply Theorem \ref{cm.crit.thm} to conclude $\o_{B_i}(-B_{ij})$ is CM.
	Once we know $B_{ij}$ is CM, we deduce that $B_{ij}$ is normal so $\big(B_{ij}, \diff_{B_{ij}}(\Gamma_i-B_{ij})\big)$ is dlt.
	We can thus proceed with a straightforward induction process to show that all strata of $\lfloor{\Delta \rfloor}$ are Cohen-Macaulay. Finally, the proof of \cite[4.16]{kk-singbook} now carries over to show that the lc centers of $(X, \Delta)$ are exactly the strata of $\lfloor{\Delta \rfloor}$. 
	We are only left to prove that the strata have rational singularities. We start by showing it for codimension one strata $B_i$. As above we can show $R^m (g|_{B_i'})_* \o_{B_i'} \text{ and } R^m (g|_{B_i'})_* \o_{B'_i} (K_{B'_i})$ vanish for $m>0$ and so we conclude $B_i$ are rational. We continue by induction to conclude.
\end{proof}

\begin{rem}
	The natural inductive proof of (4) would be to show that G-R vanishing holds over $\bigl(B_i, \diff_{B_i}(\Delta-B_i)\bigr)$. This is quite likely true, but it is not obvious how to prove it. 
	To see what the difficulty is, fix a log resolution $g \colon X'\to X$ and let $B_i'\subset X'$ denote the birational transform of $B_i$.  Given $D' \sim_{g, \r} K_{X'}+B_i'+\Delta'$ where $g_*\Delta' \leq \Delta$ and $\rdown{\ex(\Delta')}=0$, we get that
	$R^i(g|_{B_i'})_*\o_{B'}(D_i'|_{B'})=0$ for $i>0$.
	Not every log resolution of $\bigl(B_i, \diff_{B_i}(\Delta-B_i)\bigr)$ is of this form, but, as we discussed in Section~\ref{bir.inv.subsect}, this does not matter. The problem is that usually not every divisor $D_{B_i} \sim_{g|_{B', \r}} K_{B_i'}+\Delta'$ on $B_i'$ is the restriction of a divisor from $X'$ and we do not know how to overcome this.  
	However, in many applications we need G-R vanishing not for all divisors, but only for some that are constructed from $(X, \Delta)$ in a `natural' way. If we are lucky then the construction of these divisors commutes with restriction, and the proof goes through. 
	This is exactly the case  with the proof of (4). 
\end{rem}

\begin{rem}
	A general result on rationality of dlt singularities in every dimension similar to Theorem \ref{g-r.impl.other.thm}(2) has been proven in \cite[Corollary 10.18]{kov-rtl}, where it is proved that $X$ has rational singularities if it is Cohen-Macaulay and $X$ is potentially dlt.  
\end{rem}

\section{Grauert-Riemenschneider vanishing theorem}\label{g-r.sec}

The aim of this section is to prove G-R vanishing for dlt threefolds.

\begin{proof}[Proof of Theorem~\ref{dlt.chr7.gr.cor}]
	\label{dlt.chr7.gr.cor.pf}
	Since the statement of the theorem is local, we can assume $X$ to be affine.
	Let $g \colon Y \to (X,\Delta)$ be a projective log resolution such that the exceptional divisor $E:=E_1+\dots+E_n$ supports a $g$-ample $\q$-divisor whose existence is guaranteed by \cite[Corollary 3]{k-wit}.
	By \cite[Theorem 9]{k-notqfacmmp}, (see also \cite{mixed-char-3-mmp}), we can run a relative MMP starting with 
	$(X^0, \Theta^0):=(Y, \Theta:=g_*^{-1} \Delta_Y+\sum_{i=1}^n E_i)$ over $X$ as in \cite[Theorem 2]{k-notqfacmmp}.
	Let 
	$$
	\phi_i \colon (X^{i-1}, \Theta^{i-1}) \map (X^i, \Theta^i)
	$$ be the $i$-th step of the MMP and
	$g_i\colon X^i\to Y$ the corresponding map to $Y$.
	
	As we proved in Section~\ref{bir.inv.subsect}, G-R vanishing holds over the snc pair   $(X^0, \Theta^0)$. We show the following:\\
	
	{\it Claim \ref{dlt.chr7.gr.cor.pf}.1.}  If  G-R vanishing holds {\em over} $(X^{i-1}, \Theta^{i-1})$ then it  also holds {\em over} $(X^{i}, \Theta^{i})$. 
	\medskip
	
	At the end of the MMP now we conclude that  G-R vanishing holds over $(X^{m}, \Theta^{m})$. 
	Therefore it remains to show Claim~\ref{dlt.chr7.gr.cor.pf}.1 for each MMP step.
	If $\phi_i \colon (X^{i-1}, \Theta^{i-1}) \map (X^i, \Theta^i)$ is a
	divisorial contraction that contracts a  divisor $S_{i-1}\subset X^{i-1}$, then we use Proposition~\ref{g-r.desc.contr.lem}.
	For this we need to check that assumption  (\ref{g-r.desc.contr.lem}.3) is satisfied. If  $S_{i-1}$ is contracted to a curve, then K-V vanishing holds by \cite[Theorem 3.3]{Tan18}. If $S_{i-1}$  is mapped to a closed point, then  we use \cite[Theorem 1.1]{ABL20}.  
	This is where we use the assumption that the residue fields of closed points are perfect of characteristic $p> 5$.
	
	If $\phi_i \colon (X^{i-1}, \Theta^{i-1}) \map (X^i, \Theta^i)$ is a flip, then we apply Proposition~\ref{g-r.desc.flip.lem}, which works for any excellent 3 dimensional scheme. 
\end{proof}

For divisorial contractions we have the following general result.

\begin{prop} \label{g-r.desc.contr.lem}  Let $(X, S+\Delta)$ be an excellent dlt pair where $S$ is a prime divisor, $g \colon X\to Z$ a proper birational morphism contracting $S$.
	Assume that
	\begin{enumerate}
		\item  G-R vanishing holds over   $(X, S+\Delta)$,
		\item $-S$ is $g$-ample,   and
		\item  K-V vanishing holds  for 
		$g|_S:(S, \diff_S\Delta)\to g(S)$ (as in Definition~\ref{flex.KV.defn}). 
	\end{enumerate}
	Then G-R vanishing holds over $\bigl(Z, g_*\Delta\bigr)$.
\end{prop}

\begin{proof} \label{g-r.desc.contr.lem.pf}
	We use Proposition~\ref{flex.KV.ind.lem}. 
	Choose any log resolution $\pi\colon Y\to X$.
	Assume that we have $D_Y\sim_{\mathbb{R}} K_Y+\Delta'_Y$. 
	Set $D_X:=\pi_*(D_Y)$. 
	Then $\pi_*\o_Y(D_Y)=\o_X(D_X)$ by Claim \ref{c.pushforward.fujita}. 
	Since $(X, S+\Delta)$ satisfies G-R vanishing,  Theorem~\ref{g-r.impl.other.thm} shows that  $S$ is normal and $\mathcal{O}_X(D_X-mS)$ is $S_3$ along $S$ since $(D_X-mS+(S+\Delta-\Delta'_X))$ is $\r$-Cartier.
	
	Set $c:=\coeff_S\Delta'_Y$. Since $S$ is $g$-exceptional, we have $0\leq c<1$, hence
	$$
	D_X\sim_{\mathbb{R}}  K_X+\Delta'_X=K_X+S+(\Delta'_X-cS)+(1-c)(-S)
	$$
	shows that  assumption (\ref{flex.KV.ind.lem}.1) and (\ref{flex.KV.ind.lem}.4) are also satisfied.
	Thus  $R^ig_*\o_X(D_X)=0$ for $i>0$ near $g(S)$ and the Leray spectral sequence gives that $R^i(g\circ\pi)_*\o_Y(D_Y)=0$ for $i>0$ near $g(S)$. 
\end{proof}

Next we show that G-R vanishing is preserved by many 3-dimensional flips. 

\begin{prop} \label{g-r.desc.flip.lem}  Let $(X, S+\Delta)$ be a 3-dimensional excellent dlt pair and
	$$
	(C\subset X,S+\Delta)\stackrel{g}{\longrightarrow} Z\stackrel{\hphantom{aa}g^+}{\longleftarrow} (C^+ \subset X^+ , S^++\Delta^+)
	$$ 
	a flip.
	Assume that
	\begin{enumerate}
		\item  G-R vanishing holds over   $(X, S+\Delta)$, and
		\item $-S$ is $g$-ample.
	\end{enumerate}
	Then G-R vanishing holds over $( X^+ , S^++\Delta^+)$.
\end{prop}

\begin{proof} \label{g-r.desc.flip.lem.pf}
	Choose a common log resolution
	$\pi\colon Y\to X$  and  $\pi^+ \colon Y\to X^+$ and assume we have  have $D_Y\sim_{\mathbb{R}} K_Y+\Delta'_Y$. 
	Following the proof of Proposition \ref{g-r.desc.contr.lem}, we show that G-R vanishing holds over $\bigl(Z, g_*(S+\Delta)\bigr)$.
	Now we show how to deduce G-R vanishing on $(X^+, S^+ + \Delta^+)$.
	
	By G-R vanishing for excellent surfaces, $R^i\pi^{+}_* \mathcal{O}_Y(D_Y)$ is supported on closed points for $i\geq 1$. Since the dimension of the exceptional locus of $g^+$ is one, we conclude also that $R^ig^+_*(\pi^+_* \mathcal{O}_Y(D))=0$ for $i \geq 2$. As G-R vanishing holds on on $(X^+, S^+ + \Delta^+)$  we can apply Lemma~\ref{higher.Ri.lem} to $F:=\o_Y(D_Y)$ and $W:=X^+$ to conclude that $R^i \pi^+_* \mathcal{O}_Y(D_Y)=0$ for $i>0$.
\end{proof}

\begin{lem} \label{higher.Ri.lem}
	Let $\tau \colon Y\to W$ and $\pi\colon W\to Z$ be proper morphisms and $F$ a coherent sheaf on $Y$. Assume that
	\begin{enumerate}
		\item  $\dim\supp R^i\tau_*F\leq 0$ for $i\geq 1$ and 
		\item $R^i\pi_*(\tau_*F)=0$  for  $i\geq 2$. 
	\end{enumerate}
	Then the  following are equivalent.
	\begin{enumerate}\setcounter{enumi}{2}
		\item $R^i(\pi\circ\tau)_*F=0$ for $i\geq 1$.
		\item $R^i\tau_*F=0$  for $i\geq 1$ and $R^1\pi_*(\tau_*F)=0$.
	\end{enumerate}
\end{lem}

\begin{proof} 
	The only nonzero terms in the Leray spectral sequence are
	$\pi_*  R^i\tau_*F$ and $R^j\pi_*(\tau_*F)$ for $j=0,1$. This gives that
	$$
	R^i(\pi\circ\tau)_*F=\pi_*  R^i\tau_*F \qtq{for $i\geq 2$}
	$$
	and we have an exact sequence
	$$
	0\to  R^1\pi_*(\tau_*F)\to R^1(\pi\circ\tau)_*F\to \pi_*  R^1\tau_*F\to 0.\qedhere
	$$
\end{proof}

\section{Applications to Mori fiber spaces}\label{s-app}

In this section we describe some applications of the vanishing theorems proven in Section \ref{g-r.sec} to the birational geometry of threefolds in positive characteristic.

\subsection{Vanishing for log Fano contractions}

We prove the vanishing of higher direct images of the structure sheaf for log Fano contractions in characteristic $p \geq 7$.

\begin{thm}\label{t-vanishing-higher}
	Let $k$ be a perfect field of characteristic $p \geq 7$.
	Let $f \colon X \rightarrow Z$ be a proper contraction morphism between quasi-projective normal varieties over $k$.
	Suppose that there exists an effective $\q$-divisor $\Delta \geq 0$ such that
	\begin{enumerate}
		\item $(X, \Delta)$ is a klt threefold pair;
		\item  $-(K_X+\Delta)$ is $f$-big and $f$-nef;
		\item $\dim(Z) \geq 1$.
	\end{enumerate} 
	Then the natural map $\o_Z \rightarrow \mathbf{R}f_*\o_X $ is an isomorphism.
\end{thm}

\begin{proof}
	It is sufficient to prove $R^if_*\o_X=0$ for $i>0$.
	By \cite[Theorem 3.3]{Tan18} and \cite[Corollary 1.8]{BT19}, there exists an open subset $U \subset Z$ such that $(R^i f_* \o_X)|_{U}=0$ and $Z \setminus U$ is a finite set of points.
	Let $z$ be a closed point in $Z \setminus U$.
	By \cite[Proposition 2.15]{GNT19},
	there exists a contraction
	$\pi \colon Y \to Z$
	and an effective $\q$-divisor $\Delta_Y$ on $Y$ such that 
	\begin{enumerate}
		\item[(i)] $(Y, \Delta_Y)$ is a $\q$-factorial plt pair (in particular, $Y$ is klt),
		\item[(ii)] $S:=(\pi^{-1}(z))_{\text{red}}$ is an irreducible component of $\lfloor{ \Delta_Y \rfloor}$,
		\item[(iii)] $-(K_Y+\Delta_Y)$ is $\pi$- ample and $-S$ is $\pi$-nef.
	\end{enumerate}
	Let us consider the following diagram
	\[
	\begin{CD}
		W @> \varphi >> Y \\
		@V \psi VV @VV \pi V\\
		X @> f >> Z,
	\end{CD}
	\]
	where $\varphi$ and $\psi$ are log resolutions.
	Since $X$ and $Y$  have rational singularities by \cite[Corollary 1.3]{ABL20}, to prove $R^i f_* \o_X=0$ in a neighbourhood of $z$ it is sufficient to prove $R^i \pi_* \o_Y=0$ for $i>0$. 
	This follows from Proposition \ref{flex.KV.ind.lem} and \cite[Theorem 1.1]{ABL20}.
\end{proof}

The lifting of birational contractions between klt pairs was treated in  \cite[Corollary 5.1]{HW19}).
Combining  Theorem~\ref{t-vanishing-higher} with 
\cite[Theorem 3.1]{CvS09} we get the following analog 
for log Fano contractions.

\begin{cor}\label{lifting}
	Let $k$ be a perfect field of characteristic $p\geq 7$.
	Let $(X, \Delta)$ be a klt threefold pair and let $f \colon X \to Z$ be a proper contraction morphism over $k$ between quasi-projective varieties.
	Assume that
	\begin{enumerate}
		\item $X$ lifts to $W_m(k)$ (resp. formally lifts to $W(k)$),
		\item $-(K_X+\Delta)$ is $f$-big and $f$-nef,
	\end{enumerate} 
	Then the morphism $f \colon X \to Z$ lifts to $W_m(k)$ (resp. formally lifts to $W(k)$). \qed
\end{cor}

\subsection{Threefold conic bundles} \label{ss-rat-conic}

By \cite[Theorem 3.8]{NT20}, the  base scheme of a threefold conic bundle has $W\o$-rational singularities over perfect fields of characteristic $p \geq 7$.
The same holds for $p=5$ using 
\cite{HW195}.
We show that  the base of a threefold conic bundle has rational singularities for $p \geq 7$.

\begin{thm}
	Let $k$ be a perfect field of characteristic $p \geq 7$.
	Let $(X, \Delta)$ be a klt threefold log pair and let $\pi \colon X \to S$ be a proper contraction morphism onto a surface $S$ over $k$.
	If $-(K_X+\Delta)$ is $\pi$-big and $\pi$-nef, then $S$ has rational singularities. 
\end{thm}

\begin{proof}
	By \cite[Corollary 1.3]{ABL20}, $X$ has rational singularities and $\o_S \to \textbf{R}\pi_*\o_X$ is a quasi-isomorphism by Theorem \ref{t-vanishing-higher}.
	Since $S$ is a normal surface, it is  Cohen-Macaulay and thus we conclude by Proposition \ref{rationalsing}.
\end{proof}

It would be natural to expect that $S$ has klt singularities, at least for sufficiently large $p>0$.
It is easy to see  that $S$ is smooth if the total space is smooth, see \cite[4.11.2]{Kol91}.

\subsection{Pathological examples in higher dimension} \label{ss-bad-examp}

The aim of this section is to show new examples of Mori fibre spaces whose bases have non-lc singularities in characteristic $p>0$.
The first examples where the singularities of the bases were non-klt (but still log canonical) were constructed for $p=2,3$  in \cite[Theorem 1.1]{Tan20}.
Our construction is based on the work of Yasuda on wild quotient singularities  \cite{Yas14, Yas19}.

Fix an algebraically closed field $k$ of characteristic $p>0$ and the cyclic group $C_p:= \mathbb{Z}/p\mathbb{Z}$. The Jordan normal form theorem says that every $C_p$-representation decomposes as a direct sum of indecomposable representations
$V= \bigoplus_i V_i$ where $\dim V_i\leq p$. 
Yasuda introduces the invariant
$$
D_V:=\tsum_i \tbinom{\dim V_i}{2},
$$
and proves that if  $D_V \geq 2$, then the quotient variety $X:=V/C_p$ is terminal (resp. canonical, log canonical) if and only if $D_V >p$ (resp. $D_V \geq p$, $D_V \geq p-1$).

\begin{exmp}
	Assume  that $p\geq 5$ and  let $V_3$ denote the indecomposable $C_p$-representation of dimension $3$. Then  $Y:= V_3/ C_p$ is not log canonical.

	Next consider  $(\mathbb{P}^1_k)^{p}$ with the cyclic permutation  $C_p$-action.
	Set  $X:=\bigl((\mathbb{P}^1)^{p} \times V_3\bigr)/C_p$ where the action of $C_p$ on $\bigl((\mathbb{P}^1)^{p} \times V_3\bigr)$ is the diagonal one.
	At the fixed points, in local affine charts, the action is the sum of two irreducible representations $V_p \oplus V_3$. Thus $X$ has terminal singularities by Yasuda's  theorem. 
	
	Note finally that 
	$\pic\bigl((\mathbb{P}^1)^{p} \bigr)^{C_p} \simeq \mathbb{Z}.$
	We thus conclude that the coordinate projection
	$$
	\pi: X=\bigl((\mathbb{P}^1)^{p} \times V_3\bigr)/C_p\to
	V_3/C_p=Y
	$$
	is a Mori fiber space, $X$ has 
	terminal singularities and $Y$ is not even log canonical.
	
	Note that $Y$ is a hypersurface singularity by \cite[Example 6.23]{Yas14} while $X$ is not Cohen-Macaulay by \cite{ES80}.
\end{exmp}

\bibliographystyle{amsalpha}

\providecommand{\bysame}{\leavevmode\hbox to3em{\hrulefill}\thinspace}
\providecommand{\MR}{\relax\ifhmode\unskip\space\fi MR }
\providecommand{\MRhref}[2]{%
	\href{http://www.ams.org/mathscinet-getitem?mr=#1}{#2}
}
\providecommand{\href}[2]{#2}

\end{document}